\documentclass[12pt,twoside]{amsart}
\usepackage{amsmath, amsthm, amscd, amsfonts, amssymb, graphicx}
\usepackage{enumerate}
\usepackage[colorlinks=true,
linkcolor=blue,
urlcolor=cyan,
citecolor=red]{hyperref}
\usepackage{mathrsfs}
\newtheorem{theorem}{Theorem}[section]
\newtheorem{lemma}{Lemma}[section]
\newtheorem{remark}{Remark}[section]
\newtheorem{definition}{Definition}[section]
\newtheorem{corollary}{Corollary}[section]
\newtheorem{example}{Example}[section]
\newtheorem{proposition}{Proposition}[section]
\numberwithin{equation}{section}

\begin{document}
	
\title{A convex treatment of numerical radius inequalities}
\author{Zahra Heydarbeygi, Mohammad Sababheh and Hamid Reza Moradi}

\begin{abstract}
In this article, we prove an inner product inequality for Hilbert space operators. This inequality will be utilized to present a general numerical radius inequality using convex functions. Applications of the new results include obtaining new forms that generalize and extend some well known results in the literature, with an application to the newly defined generalized numerical radius.

We emphasize that the approach followed in this article is different from the approaches used in the literature to obtain such versions. 
\end{abstract}
\maketitle
\pagestyle{myheadings}
\markboth{\centerline {Z. Heydarbeygi, M. Sababheh \& H. R. Moradi}}
{\centerline {A convex treatment of numerical radius inequalities}}
\bigskip
\bigskip
\section{Introduction}
Let $\mathbb{B}\left( \mathscr{H} \right)$ denote the ${{C}^{*}}$-algebra of all bounded linear operators on a complex Hilbert space $\mathscr{H}$ with inner product $\left\langle \cdot,\cdot \right\rangle $. For $T\in \mathbb{B}\left( \mathscr{H} \right)$, let $\omega \left( T \right)$ and $\left\| T \right\|$ denote the numerical radius and the operator norm of $T$, respectively. Recall that $\omega \left( T \right)=\underset{\left\| x \right\|=1}{\mathop{\underset{x\in \mathscr{H}}{\mathop{\sup }}\,}}\,\left| \left\langle Tx,x \right\rangle  \right|$ and $\left\| T \right\|=\underset{\left\| x \right\|=1}{\mathop{\underset{x\in \mathscr{H}}{\mathop{\sup }}\,}}\,\left\| Tx \right\|$. It is evident	 that $\omega \left( \cdot \right)$ defines a norm
on $\mathbb{B}\left( \mathscr{H} \right)$, which is equivalent to the operator norm $\left\| \cdot \right\|$. In fact, for every $T\in \mathbb{B}\left( \mathscr{H} \right)$,
\begin{equation}\label{38}
\frac{1}{2}\left\| T \right\|\le \omega \left( T \right)\le \left\| T \right\|.
\end{equation}
The inequalities in \eqref{38} are sharp. The first inequality becomes an equality if ${{T}^{2}}=0$, while the second inequality becomes an equality if $T$ is normal, i.e., ${{T}^{*}}T=T{{T}^{*}}$, where $T^*$ is the adjoint operator of $T$.

In \cite{1}, Kittaneh improved the second inequality in \eqref{38} as follows
\begin{equation}\label{37}
\omega \left( T \right)\le \frac{1}{2}\left( \left\| T \right\|+{{\left\| {{T}^{2}} \right\|}^{\frac{1}{2}}} \right).
\end{equation}
The fact that \eqref{37} provides a refinement of the second inequality in \eqref{38} follows  from the fact $\|T^2\|\leq \|T\|^2.$\\
Another refinement of the second inequality in \eqref{38} was shown in \cite{8} as follows 
\begin{equation}\label{36}
{{\omega }^{2}}\left( T \right)\le \frac{1}{2}\left\| {{\left| T \right|}^{2}}+{{\left| {{T}^{*}} \right|}^{2}} \right\|, T\in \mathbb{B}\left( \mathscr{H} \right).
\end{equation}
Here $\left| T \right|$ stands for the positive operator
${{\left( {{T}^{*}}T \right)}^{\frac{1}{2}}}$.

A generalization of the inequality \eqref{36} was given in \cite{9} as follows 
\begin{equation}\label{41}
{{\omega }^{2r}}\left( T \right)\le \frac{1}{2}\left\| {{\left| T \right|}^{2r}}+{{\left| {{T}^{*}} \right|}^{2r}} \right\|, T\in \mathbb{B}\left( \mathscr{H} \right), r\geq 1.
\end{equation}
Nowadays, a considerable attention is dedicated to refinements and generalizations of the above inequalities \cite{AF, 12, 2,9, 1, 7, 8,MSS, 11, 10,zamani}.

Recent progress in this field includes sharper refinements, new refined forms and new definitions related to the numerical radius, such as the generalized numerical radius \cite{AF} and the $A-$numerical radius \cite{zamani}. We also refer the reader to the very recent papers \cite{bhu,bhu2,zamani2} for various results including better lower bounds for the numerical radius, new inequalities for the generalized numerical radius and the Davis–Wielandt radius.

Our main target in this article is to present a general form that leads to new refinements and to some already known results in the literature about the numerical radius. Our approach is based on delicate treatments of inner product inequalities via convex functions.

The main result in this paper reads as follows
\begin{equation}\label{11thm_intro}
f\left(|\left<Ax,x\right>\left<Bx,x\right>|^2\right)\leq \frac{f\left(|\left<BAx,x\right>|^2\right)+\left<\left(\alpha f\left(|A|^{\frac{2}{\alpha}}\right)+(1-\alpha)f\left(|B^*|^{\frac{2}{1-\alpha}}\right)\right)x,x\right>}{2},
\end{equation}
for $0\le \alpha \le 1$, where $A,B\in \mathbb{B}\left(\mathscr{H} \right),x \in \mathscr{H}$ is a unit vector and $f:[0,\infty)\to\mathbb{R}$ is an increasing convex function. Then upon selecting certain functions, we obtain new explicit inequalities for the numerical radius. For example, if  $r\ge 1$, the above inequality leads to the numerical radius inequality
\begin{equation}\label{39}
{{\omega }^{2r}}\left( {{B}^{*}}A \right)\le \frac{1}{2}{{\omega }^{r}}\left( {{\left| B \right|}^{2}}{{\left| A \right|}^{2}} \right)+\frac{1}{4}\left\| {{\left| A \right|}^{4r}}+{{\left| B \right|}^{4r}} \right\|,\;A,B\in \mathbb{B}\left(\mathscr{H} \right).
\end{equation}
Then we will show how this refines some results in the literature. Several applications will be presented also.

The importance of the current work lies in the fact that \eqref{11thm_intro} can be used to retrieve several interpolated inequalities for the numerical radius. These interpolated inequalities then can be used to obtain explicit forms of such inequalities.

\section{Preliminary lemmas} 
In this short section, we present some lemmas that we shall need in our analysis. The first lemma is a simple consequence of the classical Jensen and Young inequalities.
\begin{lemma}\label{22}
For $a,b\ge0$, $0\le \alpha \le 1$, and $r\ge 1$,
\[{{a}^{\alpha }}{{b}^{1-\alpha }}\le \alpha a+\left( 1-\alpha  \right)b\le {{\left( \alpha {{a}^{r}}+\left( 1-\alpha  \right){{b}^{r}} \right)}^{\frac{1}{r}}}.\]
\end{lemma}
The second lemma follows from the spectral theorem for positive operators
and Jensen's inequality (see e.g., \cite[Theorem 1.4]{4}).
\begin{lemma}\label{23}
Let $T\in \mathbb{B}\left(\mathscr{H} \right)$ be a self adjoint operator and let $x \in \mathscr{H}$ be a unit vector. If $f$ is a convex function on an interval containing the spectrum of $T$, then
\begin{equation}\label{15}
f\left(\left\langle Tx,x \right\rangle \right)\le \left\langle f(T)x,x \right\rangle. 
\end{equation}
If $f$ is concave, then \eqref{15} holds in the reverse direction. 
\end{lemma}

The third lemma is known as the mixed Schwarz inequality (see, e.g., \cite[pp. 75--76]{6}).
\begin{lemma}\label{16}
Let $T\in \mathbb{B}\left(\mathscr{H} \right)$ and let $x \in \mathscr{H}$ be a unit vector. Then,
\[{{\left| \left\langle Tx,x \right\rangle  \right|}^{2}}\le \left\langle \left| T \right|x,x \right\rangle \left\langle \left| {{T}^{*}} \right|x,x \right\rangle.\]
\end{lemma}

The fourth lemma has been shown in \cite[(18)]{7}, and is considered as a refined triangle inequality for positive operators.
\begin{lemma}\label{43}
Let $T\in \mathbb{B}\left(\mathscr{H} \right)$. Then,
\[\left\| {{\left| T \right|}^{2}}+{{\left| {{T}^{*}} \right|}^{2}} \right\|\le \left\| {{T}^{2}} \right\|+{{\left\| T \right\|}^{2}}.\]
\end{lemma}

The fifth lemma, which can be found in \cite[Theorem 2.3]{aujla}, gives a norm inequality involving convex function of positive operators.
\begin{lemma}\label{3}
Let $f$ be a non-negative nondecreasing convex function on $\left[ 0,\infty  \right)$ and let $A,B\in \mathbb{B}\left(\mathscr{H} \right)$ be positive operators. Then
\[\left\| f\left( \frac{A+B}{2} \right) \right\|\le \left\| \frac{f\left( A \right)+f\left( B \right)}{2} \right\|.\]
\end{lemma}

\section{Main Results}\label{42}
In this section, we present our main results. However, we present these results in consecutive subsections, where an inner product inequality for Hilbert space operators is shown via convex functions in the first subsection. Then applications of this generalized form are presented in the second and third subsections.
\subsection{ Inner product inequalities}
Our first main result can be stated as follows.
\begin{theorem}\label{14}
Let $A,B\in \mathbb{B}\left(\mathscr{H} \right)$ and let $x \in \mathscr{H}$ be a unit vector. If $f:[0,\infty)\to\mathbb{R}$ is an increasing convex function, then  
\begin{equation}\label{11thm}
f\left(|\left<Ax,x\right>\left<Bx,x\right>|^2\right)\leq \frac{f\left(|\left<BAx,x\right>|^2\right)+\left<\left(\alpha f\left(|A|^{\frac{2}{\alpha}}\right)+(1-\alpha)f\left(|B^*|^{\frac{2}{1-\alpha}}\right)\right)x,x\right>}{2},
\end{equation}
for $0\le \alpha \le 1$. Further,
\begin{equation}\label{29thm}
f\left(|\left<Ax,x\right>\left<Bx,x\right>|\right)\leq \frac{1}{2} f\left(|\left<BAx,x\right>|\right)+\frac{1}{4}\left<(f(|A|^2)+f(|B^*|^2))x,x\right>.
\end{equation}
\end{theorem}
\begin{proof}
In \cite{5}, the following refinement of the Cauchy-Schwarz  inequality  was shown
\[\left| \left\langle a,b \right\rangle  \right|\le \left| \left\langle a,e \right\rangle \left\langle e,b \right\rangle  \right|+\left| \left\langle a,b \right\rangle -\left\langle a,e \right\rangle \left\langle e,b \right\rangle  \right|\le \left\| a \right\|\left\| b \right\|,\]
where $a$, $b$, $e$ are vectors in $\mathscr{H}$ and $\left\| e \right\|=1$. Since
\[\begin{aligned}
   \left| \left\langle a,e \right\rangle \left\langle e,b \right\rangle  \right|+\left| \left\langle a,b \right\rangle -\left\langle a,e \right\rangle \left\langle e,b \right\rangle  \right|&\ge \left| \left\langle a,e \right\rangle \left\langle e,b \right\rangle  \right|-\left| \left\langle a,b \right\rangle  \right|+\left| \left\langle a,e \right\rangle \left\langle e,b \right\rangle  \right| \\ 
 & =2\left| \left\langle a,e \right\rangle \left\langle e,b \right\rangle  \right|-\left| \left\langle a,b \right\rangle  \right|,  
\end{aligned}\]
we have (see also \cite{3})
\begin{equation}\label{28}
\left| \left\langle a,e \right\rangle \left\langle e,b \right\rangle  \right|\le \frac{1}{2}\left( \left| \left\langle a,b \right\rangle  \right|+\left\| a \right\|\left\| b \right\| \right).
\end{equation}
Putting $e  =x$ with $\left\| x \right\|=1$, $a=Ax$ and $b={{B}^{*}}x$ in the inequality \eqref{28}, we obtain
\begin{eqnarray}\label{30}
\left| \left\langle Ax,x \right\rangle \left\langle Bx,x \right\rangle  \right|\le \frac{1}{2}\left( \left| \left\langle BAx,x \right\rangle  \right|+\left\| Ax \right\|\left\| {{B}^{*}}x \right\| \right).
\end{eqnarray}
Therefore,
\begin{align}
   {{\left| \left\langle Ax,x \right\rangle \left\langle Bx,x \right\rangle  \right|}^{2}}&\le {{\left( \frac{\left| \left\langle BAx,x \right\rangle  \right|+\left\| Ax \right\|\left\| {{B}^{*}}x \right\|}{2} \right)}^{2}} \nonumber\\ 
 & \le \frac{1}{2}\left( {{\left| \left\langle BAx,x \right\rangle  \right|}^{2}}+{{\left\| Ax \right\|}^{2}}{{\left\| {{B}^{*}}x \right\|}^{2}} \right) \label{5}\\ 
 & =\frac{1}{2}\left( {{\left| \left\langle BAx,x \right\rangle  \right|}^{2}}+\left\langle Ax,Ax \right\rangle \left\langle {{B}^{*}}x,{{B}^{*}}x \right\rangle  \right) \nonumber\\ 
 & =\frac{1}{2}\left( {{\left| \left\langle BAx,x \right\rangle  \right|}^{2}}+\left\langle {{\left| A \right|}^{2}}x,x \right\rangle \left\langle {{\left| {{B}^{*}} \right|}^{2}}x,x \right\rangle  \right) \nonumber\\ 
 & =\frac{1}{2}\left( {{\left| \left\langle BAx,x \right\rangle  \right|}^{2}}+\left\langle {{\left( {{\left| A \right|}^{\frac{2}{\alpha }}} \right)}^{\alpha }}x,x \right\rangle \left\langle {{\left( {{\left| {{B}^{*}} \right|}^{\frac{2}{1-\alpha }}} \right)}^{1-\alpha }}x,x \right\rangle  \right) \nonumber\\ 
 & \le \frac{1}{2}\left( {{\left| \left\langle BAx,x \right\rangle  \right|}^{2}}+{{\left\langle {{\left| A \right|}^{\frac{2}{\alpha }}}x,x \right\rangle }^{\alpha }}{{\left\langle {{\left| {{B}^{*}} \right|}^{\frac{2}{1-\alpha }}}x,x \right\rangle }^{1-\alpha }} \right) \label{6}\\ 
 & \le \frac{1}{2}\left( {{\left| \left\langle BAx,x \right\rangle  \right|}^{2}}+\alpha \left\langle {{\left| A \right|}^{\frac{2}{\alpha }}}x,x \right\rangle +\left( 1-\alpha  \right)\left\langle {{\left| {{B}^{*}} \right|}^{\frac{2}{1-\alpha }}}x,x \right\rangle  \right), \label{7}
 \end{align}
 where in \eqref{5} we have used the fact that the function $t\mapsto t^2$ is convex, in \eqref{6} we have used Lemma \ref{23} and in \eqref{7} we have used Lemma \ref{15}. 
 
 Now since $f$ is increasing and convex, \eqref{7} implies
 \begin{align*}
 f\left(|\left<Ax,x\right>\left<Bx,x\right>|^2\right)& \le  f\left(\frac{{{\left| \left\langle BAx,x \right\rangle  \right|}^{2}}+{{\left( \alpha \left\langle {{\left| A \right|}^{\frac{2}{\alpha }}}x,x \right\rangle +\left( 1-\alpha  \right)\left\langle {{\left| {{B}^{*}} \right|}^{\frac{2}{1-\alpha }}}x,x \right\rangle  \right)}}}{2}\right) \\ 
 &\leq\frac{f\left(\left|\left<BAx,x\right>\right|^2\right)+f{{\left( \alpha \left\langle {{\left| A \right|}^{\frac{2}{\alpha }}}x,x \right\rangle +\left( 1-\alpha  \right)\left\langle {{\left| {{B}^{*}} \right|}^{\frac{2}{1-\alpha }}}x,x \right\rangle  \right)}}}{2}\\
&\leq \frac{f\left(\left|\left<BAx,x\right>\right|^2\right)+{{ \alpha f\left(\left\langle {{\left| A \right|}^{\frac{2}{\alpha }}}x,x \right\rangle \right) +\left( 1-\alpha  \right)f\left(\left\langle {{\left| {{B}^{*}} \right|}^{\frac{2}{1-\alpha }}}x,x \right\rangle\right)}}}{2}\\
&\leq \frac{f\left(\left|\left<BAx,x\right>\right|^2\right)+ \alpha \left\langle f\left(\left| A \right|^{\frac{2}{\alpha }}\right)x,x \right\rangle  +\left( 1-\alpha  \right)\left\langle f\left(\left| B^{*} \right|^{\frac{2}{1-\alpha }}\right)x,x \right\rangle}{2}\\
 & \le \frac{f\left(|\left<BAx,x\right>|^2\right)+\left<\left(\alpha f\left(|A|^{\frac{2}{\alpha}}\right)+(1-\alpha)f\left(|B^*|^{\frac{2}{1-\alpha}}\right)\right)x,x\right>}{2},
\end{align*}
where we have used the fact that $f$ is convex and Lemma \ref{23} to obtain the above inequalities. This completes the proof of \eqref{11thm}.

On the other hand, from \eqref{30}, we infer for any unit vector $x \in \mathscr{H}$,
\begin{align*}
\left|\left<Ax,x\right>\left<Bx,x\right>\right|&\leq \frac{\left|\left<BAx,x\right>\right|+\left<Ax,Ax\right>^{1/2}\left<B^*x,B^*x\right>^{1/2}}{2}\\
&=\frac{\left|\left<BAx,x\right>\right|+\left<|A|^2x,x\right>^{1/2}\left<|B^*|^2x,x\right>^{1/2}}{2}\\
&\leq \frac{\left|\left<BAx,x\right>\right|+\frac{\left<|A|^2x,x\right>+\left<|B^*|^2x,x\right>}{2}}{2}.
\end{align*}
Again, since $f$ is increasing and convex, we obtain
\begin{align*}
f\left( \left|\left<Ax,x\right>\left<Bx,x\right>\right|   \right)&\leq f\left( \frac{\left|\left<BAx,x\right>\right|+\frac{\left<|A|^2x,x\right>+\left<|B^*|^2x,x\right>}{2}}{2}    \right)\\
&\leq  \frac{f\left(\left|\left<BAx,x\right>\right|\right)+f\left(\frac{\left<|A|^2x,x\right>+\left<|B^*|^2x,x\right>}{2}\right)}{2}\\
&\leq \frac{f\left(\left|\left<BAx,x\right>\right|\right)+\frac{f\left(\left<|A|^2x,x\right>\right)+f\left(\left<|B^*|^2x,x\right>\right)}{2}}{2}\\
&\leq \frac{f\left(\left|\left<BAx,x\right>\right|\right)+\frac{\left<f\left(|A|^2\right)x,x\right>+\left<f\left(|B^*|^2\right)x,x\right>}{2}}{2}\\
&=\frac{1}{2} f\left(|\left<BAx,x\right>|\right)+\frac{1}{4}\left<(f(|A|^2)+f(|B^*|^2))x,x\right>,
\end{align*}
which proves the  inequality  \eqref{29thm} and completes the proof of the theorem.
\end{proof}
Noting that the function $f(t)=t^{r}, r\geq 1$ satisfies the conditions in Theorem \ref{14}, we obtain the following special case.
\begin{corollary}\label{cor14}
Let $A,B\in \mathbb{B}\left(\mathscr{H} \right)$ and let $x \in \mathscr{H}$ be a unit vector. Then for any $r\ge 1$ and $0\le \alpha \le 1$,
\begin{equation}\label{11}
{{\left| \left\langle Ax,x \right\rangle \left\langle Bx,x \right\rangle  \right|}^{2r}}\le \frac{1}{2}\left( {{\left| \left\langle BAx,x \right\rangle  \right|}^{2r}}+\left\langle \left( \alpha {{\left| A \right|}^{\frac{2r}{\alpha }}}+\left( 1-\alpha  \right){{\left| {{B}^{*}} \right|}^{\frac{2r}{1-\alpha }}} \right)x,x \right\rangle  \right),
\end{equation}
and
\begin{equation}\label{29}
{{\left| \left\langle Ax,x \right\rangle \left\langle Bx,x \right\rangle  \right|}^{r}}\le \frac{1}{2}{{\left| \left\langle BAx,x \right\rangle  \right|}^{r}}+\frac{1}{4}\left\langle {{(\left| A \right|}^{2r}}+{{\left| {{B}^{*}} \right|}^{2r}})x,x \right\rangle.
\end{equation}
\end{corollary}

\subsection{Applications to numerical radius inequalities}
The first application of Theorem \ref{14} and Corollary \ref{cor14} is the following numerical radius inequality for the product of two operators.
\begin{corollary}\label{12}
Let $A,B\in \mathbb{B}\left(\mathscr{H} \right)$ and let $f:[0,\infty)\to\mathbb{R}$ be an increasing convex function. Then
\begin{align*}
f\left(\omega^2(B^*A)\right)&\leq \frac{1}{2}f\left(\omega(|B|^2|A|^2)\right)+\frac{1}{4}\left\|f(|A|^4)+f(|B|^4)\right\|.
\end{align*}

In particular, if  $r\ge 1$, then
\begin{equation}\label{39}
{{\omega }^{2r}}\left( {{B}^{*}}A \right)\le \frac{1}{2}{{\omega }^{r}}\left( {{\left| B \right|}^{2}}{{\left| A \right|}^{2}} \right)+\frac{1}{4}\left\| {{\left| A \right|}^{4r}}+{{\left| B \right|}^{4r}} \right\|.
\end{equation}
\end{corollary}
\begin{proof}
Replacing $A$ and $B$ by ${{\left| A \right|}^{2}}$ and ${{\left| B \right|}^{2}}$ respectively in Theorem \ref{14}, then the inequality \eqref{29thm} reduces to 
\begin{align}
f\left(\left<|A|^2x,x\right>\left<|B|^2x,x\right>\right)&\leq\frac{1}{2}f\left(|\left<|B|^2|A|^2x,x\right>|\right)+\frac{1}{4}\left<\left(f(|A|^4)+f(|B|^4)\right)x,x\right>.\label{needed_1}
\end{align}

On the other hand,
\[\begin{aligned}
   {{\left| \left\langle {{B}^{*}}Ax,x \right\rangle  \right|}^{2}}&={{\left| \left\langle Ax,Bx \right\rangle  \right|}^{2}} \\ 
 & \le {{\left\| Ax \right\|}^{2}}{{\left\| Bx \right\|}^{2}} \quad \text{(by the Cauchy--Schwarz inequality)}\\ 
 & ={{\left\langle {{\left| A \right|}^{2}}x,x \right\rangle }}{{\left\langle {{\left| B \right|}^{2}}x,x \right\rangle }}.  
\end{aligned}\]
Since $f$ is increasing, it follows that $$f\left(  {{\left| \left\langle {{B}^{*}}Ax,x \right\rangle  \right|}^{2}}   \right)\leq f\left( {{\left\langle {{\left| A \right|}^{2}}x,x \right\rangle }}{{\left\langle {{\left| B \right|}^{2}}x,x \right\rangle }}   \right).$$
This together with \eqref{needed_1} imply
$$f\left(  {{\left| \left\langle {{B}^{*}}Ax,x \right\rangle  \right|}^{2}}   \right)\leq \frac{1}{2}f\left(|\left<|B|^2|A|^2x,x\right>|\right)+\frac{1}{4}\left<\left(f(|A|^4)+f(|B|^4)\right)x,x\right>,$$ which implies the first desired inequality upon taking the supremum over all unit vectors $x\in\mathscr{H}.$ The second inequality follows from the first by letting $f(t)=t^r; r\geq 1.$
\end{proof}
In \cite{2}, it was shown that

\begin{equation}\label{drag_2}
\omega^{2r}(B^*A)\leq \frac{1}{2}\|\;|A|^{4r}+|B|^{4r}\|, r\geq 1.
\end{equation}

Noting the following computations
\begin{align*}
\omega^r(|B|^2|A|^2)&\leq \left\|\;|A|^2|B|^2\right\|^r\\
&\leq \frac{1}{2}\left\|\;|A|^4+|B|^4\right\|^r\\
&=\left\|\left(\frac{|A|^4+|B|^4}{2}\right)^r\right\|\\
&\leq \frac{1}{2}\left\|\;|A|^{4r}+|B|^{4r}\right\|,
\end{align*}
we notice that \eqref{39} implies
\begin{align*}
\omega^{2r}(B^*A)&\leq \frac{1}{2}\left\|\;|A|^2|B|^2\right\|^r+\frac{1}{4}\left\|\;|A|^{4r}+|B|^{4r}\right\|\\
&\leq \frac{1}{4}\left\|\;|A|^4+|B|^4\right\|^r+\frac{1}{4}\left\|\;|A|^{4r}+|B|^{4r}\right\|\\
&\leq \frac{1}{4}\left\|\;|A|^{4r}+|B|^{4r}\right\|+\frac{1}{4}\left\|\;|A|^{4r}+|B|^{4r}\right\|\\
&=\frac{1}{2}\left\|\;|A|^{4r}+|B|^{4r}\right\|.
\end{align*}
Consequently, Corollary \ref{12} provides a refinement of \eqref{drag_2}.

In the following, we give a numerical example to show how Corollary \ref{12} provides a refinement of \eqref{drag_2}.
\begin{example}
Let $A=\left[ \begin{matrix}
   0 & 1  \\
   0 & 2  \\
\end{matrix} \right]$ and $B=\left[ \begin{matrix}
   2 & 0  \\
   1 & 0  \\
\end{matrix} \right]$. Then 
\[\begin{aligned}
   {{\omega }^{2}}\left( {{B}^{*}}A \right)=4&<\frac{1}{2}\omega \left( {{\left| B \right|}^{2}}{{\left| A \right|}^{2}} \right)+\frac{1}{4}\left\| {{\left| A \right|}^{4}}+{{\left| B \right|}^{4}} \right\|=\frac{25}{4}.
   \end{aligned}\]
 On the other hand, we have
 $$\frac{1}{2}\left\| {{\left| A \right|}^{4}}+{{\left| B \right|}^{4}} \right\|=\frac{25}{2}.$$
\end{example}

\begin{remark}
Notice that the inequality \eqref{39} is sharp. Indeed if $r=1$ and $A=B$, we get ${{\left\| A \right\|}^{4}}$ 
on both sides of \eqref{39}.
\end{remark}

\begin{remark}
In this remark, we show that Corollary \ref{12} provides a refinement of Dragomir's result \eqref{drag_2}. Notice first that
\begin{align*}
{{\omega }^{r}}\left( {{\left| B \right|}^{2}}{{\left| A \right|}^{2}} \right)&\le \left\|\;|B|^2|A|^2\right\|^r\\
&\leq \left\|\frac{\|A|^4+|B|^4}{2}\right\|^r\\
&=\left\|\left(\frac{\|A|^4+|B|^4}{2}\right)^r\right\|\\
 &\leq\frac{1}{2}\left\| {{\left| A \right|}^{4r}}+{{\left| B \right|}^{4r}} \right\|.
\end{align*}

Consequently,  Corollary \ref{12} implies that
\begin{equation}
\begin{aligned}\label{44}
   {{\omega }^{2r}}\left( {{B}^{*}}A \right)&\le \frac{1}{2}{{\omega }^{r}}\left( {{\left| B \right|}^{2}}{{\left| A \right|}^{2}} \right)+\frac{1}{4}\left\| {{\left| A \right|}^{4r}}+{{\left| B \right|}^{4r}} \right\| \\ 
 & \le \frac{1}{2}\left\| {{\left| A \right|}^{4r}}+{{\left| B \right|}^{4r}} \right\|,
\end{aligned}
\end{equation}
explaining why Corollary \ref{12} provide a refinement of the inequality \eqref{drag_2}. Further, the first inequality in Corollary \ref{12} provides a generalization of \eqref{drag_2}, using increasing convex functions.
\end{remark}

Now Theorem \ref{14} is utilized to obtain the following  numerical radius inequality for one operator.
\begin{corollary}
Let $T\in \mathbb{B}\left(\mathscr{H} \right)$ and let $f:[0,\infty)\to\mathbb{R}$ be an increasing convex function. Then for $0\leq \alpha\leq 1,$
\begin{align*}
f(\omega^4(T))&\leq \frac{1}{2}\left(f(\omega^2(|T|\;|T^*|))+\left\|(1-\alpha)f\left(|T|^{\frac{2}{1-\alpha}}\right)+\alpha f\left(|T^*|^{\frac{2}{\alpha}}\right)\right\|\right),
\end{align*}
and
\begin{align*}
f(\omega^2(T))&\leq\frac{1}{2} f(\omega(|T|\;|T^*|))+\frac{1}{4}\left\|f(|T|^2)+f(|T^*|^2)\right\|.
\end{align*}

 In particular, if $r\ge 1$, then
\begin{equation}\label{21}
{{\omega }^{4r}}\left( T \right)\le \frac{1}{2}\left( {{\omega }^{2r}}\left( \left| T \right|\left| {{T}^{*}} \right| \right)+\left\| \left( 1-\alpha  \right){{\left| T \right|}^{\frac{2r}{1-\alpha }}}+\alpha {{\left| {{T}^{*}} \right|}^{\frac{2r}{\alpha }}} \right\| \right),
\end{equation}
and
\begin{equation}\label{31}
{{\omega }^{2r}}\left( T \right)\le \frac{1}{2}{{\omega }^{r}}\left( \left| T \right|\left| {{T}^{*}} \right| \right)+\frac{1}{4}\left\| {{\left| T \right|}^{2r}}+{{\left| {{T}^{*}} \right|}^{2r}} \right\|.
\end{equation}
Both inequalities \eqref{21} and \eqref{31} are sharp.
\end{corollary}
\begin{proof}
Replacing $A=\left| {{T}^{*}} \right|$ and $B=\left| T \right|$ in the inequality \eqref{11thm}, we get
\begin{align*}
f\left(|\left<|T|x,x\right>\left<|T^*|x,x\right>|^2\right)&\leq\frac{f\left(|\left<|T|\;|T^*|x,x\right>|^2\right)+\left<\left\{(1-\alpha)f\left(|T|^{\frac{2}{1-\alpha}}\right)+\alpha f\left(|T^*|^{\frac{2}{\alpha}}\right)\right\}x,x\right>}{2}.
\end{align*}
Since $f$ is increasing, it follows from Lemma \ref{16} that
\begin{align*}
f\left(|\left<Tx,x\right>|^4\right)&\leq \frac{f\left(|\left<|T|\;|T^*|x,x\right>|^2\right)+\left<\left\{(1-\alpha)f\left(|T|^{\frac{2}{1-\alpha}}\right)+\alpha f\left(|T^*|^{\frac{2}{\alpha}}\right)\right\}x,x\right>}{2}.
\end{align*}
Taking the supremum over unit vectors $x$ implies the first desired inequality. The second inequality follows in a similar way, but using \eqref{29thm}.\\
The other two inequalities follow by by letting $f(t)=t^{r}; r\geq 1.$

To show sharpness of \eqref{21} (resp. \eqref{31}), assume that $T$ is a normal operator. For $r=1$ and $\alpha =\frac{1}{2}$, we get ${{\left\| T \right\|}^{4}}$ (resp. ${{\left\| T \right\|}^{2}}$) on both sides, completing the proof.
\end{proof}

In the following we give a numerical example calculating the terms appearing in \eqref{31}. Also, this example shows how \eqref{31} refines \eqref{36} numerically.

\begin{example}
 Let $T=\left[ \begin{matrix}
   2 & 1  \\
   0 & 1  \\
\end{matrix} \right]$. Then
\[\begin{aligned}
   {{\omega }^{2}}\left( T \right)\approx 4.87132&<\frac{1}{2}\omega \left( \left| T \right|\left| {{T}^{*}} \right| \right)+\frac{1}{4}\left\| {{\left| T \right|}^{2}}+{{\left| {{T}^{*}} \right|}^{2}} \right\|\approx 5.0712.
   \end{aligned}\]
   On the other hand, we have
 $$\frac{1}{2}\left\| {{\left| T \right|}^{2}}+{{\left| {{T}^{*}} \right|}^{2}} \right\|\approx 5.12132.  $$
\end{example}

The following result will be needed for further investigation; yet it is of interest by itself.
\begin{proposition}\label{33}
Let $T\in \mathbb{B}\left(\mathscr{H} \right)$. Then for any $r\ge 1$ and $0\le \alpha \le 1$,
\begin{equation}\label{20}
{{\omega }^{2r}}\left( \left| T \right|\left| {{T}^{*}} \right| \right)\le \left\| \left( 1-\alpha  \right){{\left| T \right|}^{\frac{2r}{1-\alpha }}}+\alpha {{\left| {{T}^{*}} \right|}^{\frac{2r}{\alpha }}} \right\|,
\end{equation}
and
\begin{equation}\label{34}
{{\omega }^{r}}\left( \left| T \right|\left| {{T}^{*}} \right| \right)\le \frac{1}{2}\left\| {{\left| T \right|}^{2r}}+{{\left| {{T}^{*}} \right|}^{2r}} \right\|.
\end{equation}
\end{proposition}
\begin{proof}
Let $x \in \mathscr{H}$ be a unit vector. We have
\begin{align}
   {{\left| \left\langle \left| T \right|\left| {{T}^{*}} \right|x,x \right\rangle  \right|}^{2r}}&={{\left| \left\langle \left| {{T}^{*}} \right|x,\left| T \right|x \right\rangle  \right|}^{2r}} \nonumber\\ 
 & \le {{\left\| \left| T \right|x \right\|}^{2r}}{{\left\| \left| {{T}^{*}} \right|x \right\|}^{2r}} \label{24}\\ 
 & ={{\left\langle \left| T \right|x,\left| T \right|x \right\rangle }^{r}}{{\left\langle \left| {{T}^{*}} \right|x,\left| {{T}^{*}} \right|x \right\rangle }^{r}} \nonumber\\ 
 & ={{\left\langle {{\left| T \right|}^{2}}x,x \right\rangle }^{r}}{{\left\langle {{\left| {{T}^{*}} \right|}^{2}}x,x \right\rangle }^{r}} \nonumber\\ 
  & \le \left\langle {{\left| T \right|}^{2r}}x,x \right\rangle \left\langle {{\left| {{T}^{*}} \right|}^{2r}}x,x \right\rangle  \label{25}\\ 
 & =\left\langle {{\left( {{\left| T \right|}^{\frac{2r}{1-\alpha }}} \right)}^{1-\alpha }}x,x \right\rangle \left\langle {{\left( {{\left| {{T}^{*}} \right|}^{\frac{2r}{\alpha }}} \right)}^{\alpha }}x,x \right\rangle  \nonumber\\ 
 & \le {{\left\langle {{\left| T \right|}^{\frac{2r}{1-\alpha }}}x,x \right\rangle }^{1-\alpha }}{{\left\langle {{\left| {{T}^{*}} \right|}^{\frac{2r}{\alpha }}}x,x \right\rangle }^{\alpha }}\label{26} \\ 
 & \le \left( 1-\alpha  \right)\left\langle {{\left| T \right|}^{\frac{2r}{1-\alpha }}}x,x \right\rangle +\alpha \left\langle {{\left| {{T}^{*}} \right|}^{\frac{2r}{\alpha }}}x,x \right\rangle  \label{27}\\ 
 & =\left\langle \left( \left( 1-\alpha  \right){{\left| T \right|}^{\frac{2r}{1-\alpha }}}+\alpha {{\left| {{T}^{*}} \right|}^{\frac{2r}{\alpha }}} \right)x,x \right\rangle,   \nonumber
\end{align}
where in the inequality \eqref{24} we have used the Cauchy--Schwarz inequality, the inequalities \eqref{25} and \eqref{26} are obtained from Lemma \ref{23}, and the inequality \eqref{27} is a consequence of the first inequality in Lemma \ref{22}.

Whence,
\begin{equation}\label{19}
{{\left| \left\langle \left| T \right|\left| {{T}^{*}} \right|x,x \right\rangle  \right|}^{2r}}\le \left\langle \left( \left( 1-\alpha  \right){{\left| T \right|}^{\frac{2r}{1-\alpha }}}+\alpha {{\left| {{T}^{*}} \right|}^{\frac{2r}{\alpha }}} \right)x,x \right\rangle,
\end{equation}
for any unit vector $x \in \mathscr{H}$. Taking the supremum over $x \in \mathscr{H}$ with $\left\| x \right\|=1$ in the inequality \eqref{19}, we obtain \eqref{20}.\\
Similar argument implies
\begin{equation}\label{46}
{{\left| \left\langle \left| T \right|\left| {{T}^{*}} \right|x,x \right\rangle  \right|}^{r}}\le \frac{1}{2}\left\langle \left( {{\left| T \right|}^{2r}}+{{\left| {{T}^{*}} \right|}^{2r}} \right)x,x \right\rangle,
\end{equation}
for any unit vector $x \in \mathscr{H}$. Taking the supremum over $x\in \mathscr{H}$, $\left\| x \right\|=1$ in \eqref{46} produces the inequality \eqref{34}.
\end{proof}

\begin{remark}
By combining inequalities \eqref{31} and \eqref{34}, we infer that
\begin{equation}\label{35}
\begin{aligned}
   {{\omega }^{2r}}\left( T \right)&\le \frac{1}{2}{{\omega }^{r}}\left( \left| T \right|\left| {{T}^{*}} \right| \right)+\frac{1}{4}\left\| {{\left| T \right|}^{2r}}+{{\left| {{T}^{*}} \right|}^{2r}} \right\| \\ 
 & \le \frac{1}{2}\left\| {{\left| T \right|}^{2r}}+{{\left| {{T}^{*}} \right|}^{2r}} \right\|.  
\end{aligned}
\end{equation}
Consequently, the inequalities \eqref{35} provide a refinement of the inequality \eqref{41}
\end{remark}

The following corollary shows that the  inequality \eqref{31} provides an improvement of the inequality \eqref{37}.
\begin{corollary}
Let $T\in \mathbb{B}\left(\mathscr{H} \right)$. Then \[\omega \left( T \right)\le \frac{1}{2}\sqrt{2\omega \left( \left| T \right|\left| {{T}^{*}} \right| \right)+\left\| {{\left| T \right|}^{2}}+{{\left| {{T}^{*}} \right|}^{2}} \right\|}\le \frac{1}{2}\left( \left\| {{T}^{2}} \right\|^{1/2}+{{\left\| T \right\|}} \right).\]
\end{corollary}
\begin{proof}
We have
\[\begin{aligned}
   \omega \left( T \right)&\le \frac{1}{2}\sqrt{2\omega \left( \left| T \right|\left| {{T}^{*}} \right| \right)+\left\| {{\left| T \right|}^{2}}+{{\left| {{T}^{*}} \right|}^{2}} \right\|} \quad \text{(by \eqref{31})}\\ 
 & \le \frac{1}{2}\sqrt{2\left\| \left| T \right|\left| {{T}^{*}} \right| \right\|+\left\| {{\left| T \right|}^{2}}+{{\left| {{T}^{*}} \right|}^{2}} \right\|} \quad \text{(by the second inequality in \eqref{38})}\\ 
 & =\frac{1}{2}\sqrt{2\left\| {{T}^{2}} \right\|+\left\| {{\left| T \right|}^{2}}+{{\left| {{T}^{*}} \right|}^{2}} \right\|} \quad \text{(since $\left\| \left| T \right|\left| {{T}^{*}} \right| \right\|=\left\| {{T}^{2}} \right\|$)}\\ 
 & \le \frac{1}{2}\sqrt{2\left\| {{T}^{2}} \right\|+\left\| {{T}^{2}} \right\|+{{\left\| T \right\|}^{2}}} \quad \text{(by Lemma \ref{43})}\\ 
 & \le \frac{1}{2}\sqrt{2\left\| T \right\|{{\left\| {{T}^{2}} \right\|}^{\frac{1}{2}}}+\left\| {{T}^{2}} \right\|+{{\left\| T \right\|}^{2}}} \quad \text{(since $\left\| {{T}^{2}} \right\|={{\left\| {{T}^{2}} \right\|}^{\frac{1}{2}}}{{\left\| {{T}^{2}} \right\|}^{\frac{1}{2}}}\le \left\| T \right\|{{\left\| {{T}^{2}} \right\|}^{\frac{1}{2}}}$)}\\ 
 & =\frac{1}{2}\sqrt{{{\left( \left\| {{T}^{2}} \right\|^{1/2}+{{\left\| T \right\|}} \right)}^{2}}} \\ 
 & =\frac{1}{2}\left( \left\| {{T}^{2}} \right\|^{1/2}+{{\left\| T \right\|}} \right),
\end{aligned}\]
and the proof is complete. 
\end{proof}

\subsection{The generalized numerical radius}\label{sec3}
In this section, we present some new inequalities for the generalized numerical radius $\omega_N(\cdot)$, based on the inner product inequalities obtained earlier. First, we recall the following definition from \cite{AF}.
\begin{definition}
Let $T\in\mathbb{B}(\mathscr{H})$ and let $N$ be any norm on $\mathbb{B}(\mathscr{H})$. Then the generalized numerical radius of $T$, induced by the norm $N$, is defined by $\omega_N(T)=\sup\limits_{\theta\in\mathbb{R}}N(\Re(e^{i\theta}T)),$ where $\Re(T)$ is the real part of the operator $T$.
\end{definition}

In the following result, we use Proposition \ref{33} to obtain a new inequality for $\omega_N(\cdot).$ This result is stated for the algebra of all $n\times n$ matrices, denoted by $\mathcal{M}_n.$ Notice that since the finite rank operators are dense in the class of compact operators in $\mathbb{B}(\mathscr{H})$, it follows that the following result is also true for any compact operator $T\in\mathbb{B}(\mathscr{H}).$
\begin{proposition}\label{prop_w_n_1}
Let $T\in\mathcal{M}_n$ and let $N(\cdot)$ be a given unitarily invariant norm on $\mathcal{M}_n$. Then for any $r\geq 1$ and $0\leq \alpha\leq 1,$
$$\omega_N(|T|\;|T^*|)\leq N\left(\left\{(1-\alpha)|T|^{\frac{2r}{1-\alpha}}+\alpha|T^*|^{\frac{2r}{\alpha}}\right\}^{\frac{1}{2r}}\right),$$ and
$$\omega_{N}\left(|T|\;|T^*|\right)\leq N\left( \left\{\frac{|T|^{2r}+|T^*|^{2r}}{2}\right\}^{1/r} \right).$$
\end{proposition}
\begin{proof}
From Proposition \ref{33}, we have
$$\left|\left<|T|\;|T^*|x,x\right>\right|^{2r}\leq \left<\left((1-\alpha)|T|^{\frac{2r}{1-\alpha}}+\alpha|T^*|^{\frac{2r}{\alpha}}\right)x,x\right>.$$

Since $|e^{i\theta}|=1,$ this implies
$$\left|\left<e^{i\theta}|T|\;|T^*|x,x\right>\right|\leq \left<\left((1-\alpha)|T|^{\frac{2r}{1-\alpha}}+\alpha|T^*|^{\frac{2r}{\alpha}}\right)x,x\right>^{\frac{1}{2r}}.$$ But since, for any operator $A$, $|\left<\Re(A)x,x\right>|\leq |\left<Ax,x\right>|,$ it follows that
$$\left|\left<\Re\left\{e^{i\theta}|T|\;|T^*|\right\}x,x\right>\right|\leq \left<\left((1-\alpha)|T|^{\frac{2r}{1-\alpha}}+\alpha|T^*|^{\frac{2r}{\alpha}}\right)x,x\right>^{\frac{1}{2r}}.$$ By the minimax principle, it follows that, for $1\leq k\leq n,$
\begin{align*}
 s_k\left(\Re\left\{e^{i\theta}|T|\;|T^*|\right\}\right)&\leq s_k^{\frac{1}{2r}}\left((1-\alpha)|T|^{\frac{2r}{1-\alpha}}+\alpha|T^*|^{\frac{2r}{\alpha}}\right)\\
 &=s_k\left\{ \left((1-\alpha)|T|^{\frac{2r}{1-\alpha}}+\alpha|T^*|^{\frac{2r}{\alpha}}\right)   ^{\frac{1}{2r}}\right\}.
 \end{align*}
 This latter inequality implies that for any unitarily invariant norm $N(\cdot)$ on $\mathcal{M}_n$,
 $$N\left( \Re\left\{e^{i\theta}|T|\;|T^*|\right\}  \right)\leq N\left\{ \left((1-\alpha)|T|^{\frac{2r}{1-\alpha}}+\alpha|T^*|^{\frac{2r}{\alpha}}\right)   ^{\frac{1}{2r}}\right\},$$ which implies 
 $$\omega_N(|T|\;|T^*|)\leq N\left(\left\{(1-\alpha)|T|^{\frac{2r}{1-\alpha}}+\alpha|T^*|^{\frac{2r}{\alpha}}\right\}^{\frac{1}{2r}}\right)$$ upon taking the supremum over $\theta.$ This proves the first desired inequality.

The second inequality can be shown similarly, and hence we leave its proof to the reader.
\end{proof}

\begin{remark}
Notice that when $N$ is the operator norm,
$$N\left( \left\{\frac{|T|^{2r}+|T^*|^{2r}}{2}\right\}^{1/r} \right)=\left\|\left\{\frac{|T|^{2r}+|T^*|^{2r}}{2}\right\}^{1/2r}\right\|=\left\|\frac{|T|^{2r}+|T^*|^{2r}}{2}\right\|^{1/r}.$$
So, when $N(\cdot)=\|\cdot\|,$ Proposition \ref{prop_w_n_1} implies
$$w^{r}(|T|\;|T^*|)\leq \left\|\frac{|T|^{2r}+|T^*|^{2r}}{2}\right\|;$$ which has been shown earlier in Proposition \ref{33}.
\end{remark}

\vskip 0.3 true cm

{\tiny (Z. Heydarbeygi) Department of Mathematics, Payame Noor University (PNU), P.O. Box 19395-4697, Tehran, Iran.}

{\tiny \textit{E-mail address:} zheydarbeygi@yahoo.com}

{\tiny \vskip 0.3 true cm }

{\tiny (M. Sababheh) Department of Basic Sciences, Princess Sumaya University For Technology, Al Jubaiha, Amman 11941, Jordan.}

{\tiny \textit{E-mail address:} sababheh@psut.edu.jo}

{\tiny \vskip 0.3 true cm }

{\tiny (H. R. Moradi) Department of Mathematics, Payame Noor University (PNU), P.O. Box 19395-4697, Tehran, Iran.}

{\tiny \textit{E-mail address:} hrmoradi@mshdiau.ac.ir }

\end{document}